\pdfoutput=1
\RequirePackage{ifpdf}
\ifpdf 
\documentclass[pdftex]{sigma}
\else
\documentclass{sigma}
\fi

\numberwithin{equation}{section}

\newtheorem{Theorem}{Theorem}[section]
\newtheorem*{Theorem*}{Theorem}

\newtheorem{Lemma}[Theorem]{Lemma}
\newtheorem{Proposition}[Theorem]{Proposition}
\newtheorem*{main}{Main Proposition}
 { \theoremstyle{definition}
\newtheorem{Definition}[Theorem]{Definition}

\newtheorem{Remark}[Theorem]{Remark} }

\newcommand{\g}{\mathfrak{g}}

\newcommand{\R}{\mathbb{R}}

\newcommand{\N}{\mathbb{N}}

\newcommand{\E}{\mathbb{E}}

\begin{document}
\allowdisplaybreaks

\newcommand{\arXivNumber}{2201.00128}

\renewcommand{\PaperNumber}{058}

\FirstPageHeading

\ShortArticleName{Systolic Inequalities for Compact Quotients of Carnot Groups with Popp's Volume}

\ArticleName{Systolic Inequalities for Compact Quotients\\ of Carnot Groups with Popp's Volume}

\Author{Kenshiro TASHIRO}

\AuthorNameForHeading{K.~Tashiro}

\Address{Department of Mathematics, Tohoku University, Sendai Miyagi 980-8578, Japan}
\Email{\href{mailto:kenshiro.tashiro.b2@tohoku.ac.jp}{kenshiro.tashiro.b2@tohoku.ac.jp}}
\URLaddress{\url{https://sites.google.com/view/kenshiro-tashiro-homepage/home}}

\ArticleDates{Received February 10, 2022, in final form July 28, 2022; Published online August 02, 2022}

\Abstract{In this paper, we give a systolic inequality for a quotient space of a Carnot group~$\Gamma\backslash G$ with Popp's volume. Namely we show the existence of a positive constant $C$ such that the systole of $\Gamma\backslash G$ is less than ${\rm Cvol}(\Gamma\backslash G)^{\frac{1}{Q}}$, where $Q$ is the Hausdorff dimension. Moreover, the constant depends only on the dimension of the grading of the Lie algebra $\g=\bigoplus V_i$. To~prove this fact, the scalar product on $G$ introduced in the definition of Popp's volume plays a key role.}


\Classification{53C17; 26B15; 22E25}

\section{Introduction}

A sub-Riemannian manifold is a triple $(M,E,g)$,
where $M$ is a smooth manifold,
$E$ is a sub-bundle of the tangent bundle $TM$,
and $g$ is a metric on $E$.
If $E=TM$,
then it is an ordinary Riemannian manifold.
Geometry and analysis on sub-Riemannian manifolds have been actively studied in relation to differential operator theory,
geometric group theory,
control theory and optimal transport theory.
In this paper,
we study systolic inequalities on compact quotient spaces of Carnot groups.

We recall systolic inequalities on Riemannian manifolds.
On a Riemannian manifold $(M,g)$,
the \textit{systole} $\operatorname{sys}(M,g)$ is defined by
\[
\operatorname{sys}(M,g)=\inf\{l(c)\mid c\colon\text{non-contractible closed curve}\},
\]
where $l(c)$ denotes the length of $c$.
If $M$ is closed,
then the minimum exists.
A systolic inequality asserts that for a large class of closed Riemannian manifolds,
there is a constant $C>0$ such that
\[
\operatorname{sys}(M,g)\leq C\cdot \operatorname{vol}(M,g)^{\frac{1}{d}},
\]
where $d=\dim M$ and $\operatorname{vol}(M,g)=\int_{M}\mathrm{d}\mu_g$ is the total volume with respect to the natural volume form $\mathrm{d}\mu_g$.
A constant $C$ may depend on the topological type such as the dimension or the genus of the surface.

As an example,
let us consider a flat torus $\Gamma\backslash \E^d$,
where $\Gamma$ is a lattice in the $d$-dimensional Euclidean space $\E^d$.
Then its systole $\operatorname{sys}\big(\Gamma\backslash \E^d\big)$ is equal to the minimum of the length of a~straight segment that connects two points in the lattice $\Gamma$,
and the total volume is equal to the volume of the fundamental domain of $\Gamma$.

The systolic inequality on flat tori is obtained in the following procedure,
for example see~\cite[Section~1]{gro}.
Let $p\colon\E^d\to \Gamma\backslash \E^d$ be the covering map.
Trivially there is a unique positive number $R_0>0$ such that the volume of the $R_0$-ball of $\E^d$ centered at $0$,
say $B(R_0)$,
is equal to the total volume of the flat torus $\Gamma\backslash \E^d$, i.e., $\operatorname{vol}(B(R_0))=\operatorname{vol}\big(\Gamma\backslash \E^d\big)$.
Then the restriction of the covering map $p$ to the $R_0$-ball is not injective.
This implies that there is a non-contractible closed curve in $\Gamma\backslash \E^d$ such that its length is less than or equal to $2R_0$.
Hence the systole of the flat mani\-fold~$\Gamma\backslash \E^d$ is bounded above by
\[\operatorname{sys}\big(\Gamma\backslash \E^d\big)\leq 2R_0=2\omega_d^{-\frac{1}{d}}\operatorname{vol}\big(\Gamma\backslash \E^d\big)^{\frac{1}{d}},\]
where $\omega_d$ is the volume of the unit ball of the Euclidean space $\E^d$.

Such systolic inequalities are proved for surfaces such as non-flat $2$-dimensional torus $($its proof is not in the literature, but mentioned in \cite{pu}$)$,
projective space \cite{pu},
and higher genus ones \cite{bav, gro2, heb}.
For $d$-dimensional Riemannian manifolds with $d>2$,
Gromov showed the existence of the systolic constant $C_d$ for so called essential manifolds \cite{gro2}.

We give systolic inequalities for compact quotient spaces of Carnot groups.
Let $(G,V_1,\langle\cdot,\cdot\rangle_1)$ be a Carnot group,
and $\g=\bigoplus_{i=1}^k V_i$ its grading of the Lie algebra
(see Definition \ref{defcarnot}).
We call a discrete subgroup $\Gamma$ of a simply connected nilpotent Lie group $G$ a \textit{lattice} if it is cocompact and discrete.

Since $\Gamma$ acts isometrically on $G$ from the left,
we can define the quotient sub-Riemannian metric on $\Gamma\backslash G$.
The systole on the quotient space $\Gamma\backslash G$ is defined from this quotient metric.

We denote by $\operatorname{vol}(\Gamma\backslash G)$ the integral of Popp's volume form (Definition \ref{defpopp}).
Popp's volume is the left-invariant volume form defined by the scalar product $\langle\cdot,\cdot\rangle_{\g}$ on the Lie algebra $\g$.
We~describe its construction and properties in Section \ref{sec2}.

Set $d_i=\dim V_i$ and $Q=\sum_{i=1}^kid_i$.
Notice that $Q$ is equal to the Hausdorff dimension of the Carnot group with respect to any left-invariant homogeneous geodesic distance on it \cite{mit} $($its corrected proof can be found in~\cite{bel}$)$.

As we saw in flat tori case,
the lower bound estimate of the volume of the unit ball can be expected to show the desired systolic inequality.
The volume of the unit ball of a Carnot group is more complicated than the Euclidean one.
In \cite{has},
Hassannezhad--Kokarev estimated the volume of the unit ball of corank $1$ Carnot groups,
and gave an upper bound of the volume which depends on the dimension.
It seems hard to apply their method to general Carnot groups.
We indirectly compute a lower bound of the volume to obtain the following proposition,
which implies the main theorem.

\begin{main}
 There exists a constant $D=D(d_1,\dots,d_k)>0$,
 independent of the choice of a scalar product $\langle\cdot,\cdot\rangle_1$,
 such that for any $k$-step Carnot group $(G,V_1,\langle\cdot,\cdot\rangle_1)$ with $\dim V_i=d_i$,
 the volume of the unit ball $\operatorname{vol}(B_{cc}(1))$ is greater than $D$.
\end{main}

The main proposition shall be restated as Theorem~\ref{propmain} in the body text.

Let $\langle\cdot,\cdot\rangle_i$ be the scalar product on $V_i$ which will be defined in the construction of Popp's volume,
Section~\ref{sec2}.
This scalar product on a Carnot group plays a key role to prove the main proposition.
Roughly speaking,
the main idea is to construct the $d_i$-dimensional balls $B^{d_i}(\epsilon_i)\subset (V_i,\langle\cdot,\cdot\rangle_i)$, $i=1,\dots,k$, of radius $\epsilon_i$ centered at $0$ such that
\begin{itemize}\itemsep=0pt
 \item[$1)$] the numbers $\epsilon_1,\dots,\epsilon_k$ depend only on $d_1,\dots,d_k$,
 \item[$2)$] their product $B^{d_1}(\epsilon_1)\times{}\cdots {}\times B^{d_k}(\epsilon_k)$
 is contained in the unit ball $B_{cc}(1)$ via the identification $\exp\colon \g\simeq G$.
\end{itemize}

\begin{Remark}\label{rmk11}
 In fact,
 the radius $\epsilon_1,\dots,\epsilon_k$ will be chosen so that they depend only on $d_1$ and~$k$,
 while the volume of the balls $B^{d_1}(\epsilon_1),\dots,B^{d_k}(\epsilon_k)$ need the information of the dimensions $d_1,\dots,d_k$.
\end{Remark}

The Ball--Box theorem asserts a similar statement.
Let $M$ be a sub-Riemannian manifold and~$B_{sR}(x,R)$
the $R$-ball centered at $x$.
Fix a relatively compact chart $(U,(x_1,\dots,x_n))$ around~$x$.
The Ball--Box theorem claims the existence of positive constants $c$ and~$C$,
such that for all $R>0$ with $B_{sR}(x,R)\subset \overline{U}$,
\[
\mathrm{Box}^{d_1}(cR)\times \cdots\times \mathrm{Box}^{d_k}\big(cR^k\big)\subset B_{sR}(x,R)\subset \mathrm{Box}^{d_1}(CR)\times\cdots \times \mathrm{Box}^{d_k}\big(CR^k\big),
\]
where $\mathrm{Box}^{d_i}(r)$ is the $d_i$-dimensional box of size $r$ centered at $x$ in the chart,
and $d_1,\dots,d_k$ are determined by the nilpotentization at $x$.
The constants $c$ and $C$ depend on the choice of a~chart.
Thus it is a priori impossible to control the constants $c$ and $C$ by topological/geometric invariants of $M$.

The second claim just before Remark \ref{rmk11} asserts that if we equip a Carnot group $G$ with the global coordinates induced from the scalar products $\langle\cdot,\cdot\rangle_i$'s on $\bigoplus V_i=\g\simeq G$,
then the unit ball contains the direct product of balls with their radius controlled by $d_1$, $k$.
Therefore the second claim can be regarded as a quantitative version of (a part of) the Ball--Box theorem.

The main proposition implies the main theorem of this paper.

\begin{Theorem}\label{thmmain}
There is a positive constant $C=C(d_1,\dots,d_k)$,
independent of the choice of a~scalar product $\langle\cdot,\cdot\rangle_1$,
 such that for any Carnot group $(G,V_1,\langle\cdot,\cdot\rangle_1)$ with the grading $\dim V_i=d_i$,
and any lattice $\Gamma<G$,
\[
\operatorname{sys}(\Gamma\backslash G)\leq C\cdot \operatorname{vol}(\Gamma\backslash G)^{\frac{1}{Q}}.
\]
\end{Theorem}

\begin{proof}
Let $B_{cc}(R)$ be the ball in $(G,V_1,\langle\cdot,\cdot\rangle_1)$ of radius $R$ centered at
the identity element $e$.
It is well known that the volume of the $R$-ball satisfies
\[
\operatorname{vol}(B_{cc}(R))=\operatorname{vol}(B_{cc}(1))R^Q
\]
under a Haar volume,
see (\ref{homogeneous}).
On the other hand,
there is a positive number $R_0>0$ such that $\operatorname{vol}(B_{cc}(R_0))=\operatorname{vol}(\Gamma\backslash G)$.
It implies that the systole of $\Gamma\backslash G$ is less than or equal to $2R_0$.
Combined with the main proposition,
\[
\operatorname{sys}(\Gamma\backslash G)\leq 2R_0=2\operatorname{vol}(B_{cc}(1))^{-\frac{1}{Q}}\operatorname{vol}(\Gamma\backslash G)^{\frac{1}{Q}}\leq 2D^{-\frac{1}{Q}}\operatorname{vol}(\Gamma\backslash G)^{\frac{1}{Q}}.\tag*{\qed}
\]
\renewcommand{\qed}{}
\end{proof}

An example of a Carnot group is the Euclidean space $\E^d$.
Thus Theorem \ref{thmmain} is a generalization of the systolic inequality for flat tori.
\begin{Remark}
 We can write a constant $C$ by using the Hausdorff dimension $Q$ since the numbers $d_1,\dots,d_k$ can be controlled by $Q$.
\end{Remark}

\section{Carnot groups}\label{sec2}

Let $G$ be a simply connected nilpotent Lie group,
and $\g$ its Lie algebra.
The nilpotent Lie algebra~$\g$ is said to be \textit{graded} if $\g$ has a direct sum decomposition $\g=\bigoplus_{i=1}^k V_i$ such that $[V_i,V_j]\subset V_{i+j}$.
We will identify the Lie algebra $\g$ to the tangent space $T_eG$,
and call a left-invariant vector field $v\in \g$ a \textit{horizontal vector} if it is in $V_1$.

Let $\langle\cdot,\cdot\rangle_1$ be an scalar product on $V_1$.
Then we can define a sub-Riemannian structure on $G$ as follows.
For all $g\in G$,
we define a fiber-wise scalar product $\langle\cdot,\cdot\rangle_g$ on $L_{g\ast}V_1$ by $\langle L_{g\ast}v_1,L_{g\ast}v_2\rangle_g=\langle v_1,v_2\rangle$ for any couple of horizontal vectors $v_1,v_2\in V_1$.
Then the sub-bundle $E=\bigsqcup_{g\in G}L_{g\ast}V_1$ and the inner metric
$\langle\cdot,\cdot\rangle_g$ defines a left-invariant sub-Riemannian metric.

\begin{Definition}[Carnot group]\label{defcarnot}
 A triple $(G,V_1,\langle\cdot,\cdot\rangle)$ is called a Carnot group.
\end{Definition}

We say that an absolutely continuous curve $c\colon[a,b]\to G$ is \textit{horizontal} if the derivative~$\dot{c}(t)$ is in $L_{c(t)\ast}V_1$ for a.e.\ $t\in[a,b]$.
The length of a horizontal curve is given by the integral $l(c)=\int_a^b\|\dot{c}(t)\|_{c(t)}\mathrm{d}t$,
and the distance of two points in the sub-Riemannian manifold $(G,V_1,\langle\cdot,\cdot\rangle)$ is given by the infimum of the length of horizontal curves joining these points.
We call it the \textit{Carnot--Carath\'eodory distance} and denote by $\mathsf{d}_{cc}$.
We also denote by $B_{cc}(R)$ the ball in $G$ of radius $R>0$ centered at the identity element $e$.

{\bf Dilation.}
For a positive number $t>0$ and $X_i\in V_i$, $i=1,\dots,k$,
define the Lie algebra isomorphism $\delta_t\colon\g\to\g$ by
\[
\delta_t\Bigg(\sum_{j=1}^kX_j\Bigg)=\sum_{j=1}^kt^jX_j.
\]
This isomorphism $\delta_t$ is called the \textit{dilation}.

Let $\exp\colon \g\to G$ be the exponential map.
Since the group $G$ is simply connected and nilpotent,
the exponential map is a diffeomorphism.
So we will identify group elements in $G$ to vectors in $\g$ via the exponential map.
Moreover we can regard the dilation $\delta_t$ as the Lie group automorphism.

The origin of the name dilation comes from the following property.
Let $c$ be a length minimizing curve joining two points $x,y\in G$.
Since the derivative of the curve $c$ is in the left translation of $V_1$ a.e.,
the length of the curve $\delta_t\circ c$ is equal to $t\cdot \operatorname{length}(c)$.
It implies that $\mathsf{d}_{cc}(\delta_t(x),\delta_t(y))=t\mathsf{d}_{cc}(x,y)$.
Moreover,
it also implies that the volume of the ball of radius $R$ satisfies
\begin{equation}\label{homogeneous}
 \operatorname{vol}(B_{cc}(R))=\operatorname{vol}(B_{cc}(1))R^Q,
\end{equation}
with respect to a Haar volume.

{\bf Popp's volume.}
Popp's volume is introduced by Montgomery in~\cite{mon} as a gene\-ra\-li\-za\-tion of a~Riemannian canonical volume form $\mathrm{d}\mu_g=\sqrt{\det(g)}\,\mathrm{d}x_1\wedge\cdots\wedge \mathrm{d}x_d$.
The qualitative properties of Popp's volume has been actively studied.
By the result of Barilari--Rizzi~\cite{bar},
Popp's volume is characterized as an (local) isometric invariant volume on equiregular sub-Riemannian manifolds.
Agrachev--Barilari--Boscain exhibited a condition for Popp's volume being a scalar multiple of the spherical Hausdorff volume in \cite{agr2}.
The qualitative study of Popp's volume on a Carnot group is not interesting since it is a Haar volume.
Our interest is in the quantitative properties of Popp's volume.

Popp's volume is constructed as follows.
For $i\geq 2$,
we inductively define the linear map $\phi_i\colon V_1^{\otimes i}\to V_i$ by
\begin{gather*}
\phi_2(X_1\otimes X_2)=[X_1,X_2],
\\
\phi_i(X_1\otimes \cdots \otimes X_i)=[X_1,\phi_{i-1}(X_2\otimes \cdots\otimes X_i)],\qquad i\geq 3.
\end{gather*}
Clearly the linear map $\phi_i$ is surjective.
We will shortly write
\[
[X_1,\dots,X_i]:=\phi_i(X_1\otimes\cdots\otimes X_i).
\]
Recall that on the tensor product space $V_1^{\otimes i}$,
we can define the canonical scalar product by
\[
\langle u_1\otimes\cdots\otimes u_i,v_1\otimes \cdots\otimes v_i\rangle_{\otimes i}=\prod_{j=1}^i\langle u_j,v_j\rangle_1.
\]
From the scalar product space $\big(V_1^{\otimes i},\langle\cdot,\cdot\rangle_{\otimes i}\big)$,
we set an scalar product on $V_i$ by using the following lemma.

\begin{Lemma}[{\cite[Lemma~20.3]{agr}}]
 Let $E$ be an scalar product space,
 $V$ a vector space,
 and $f{:}$ $E\to V$ a surjective linear map.
 Then $f$ induces an scalar product on $V$ such that the norm of $v\in V$~is
 \[
 \|v\|_{V}=\min\{\|u\|_E\mid f(u)=v\}.
 \]
\end{Lemma}

Applying the above lemma to the surjective linear map $\phi_i$,
we obtain the scalar product $\langle\cdot,\cdot\rangle_i$ on each $V_i$ for $i=2,\dots,k$,
and on their direct sum $\g=\bigoplus V_i$.
We denote by $\langle\cdot,\cdot\rangle_{\g}$ the induced scalar product on $\g$.
\begin{Definition}\label{defpopp}
 Popp's volume on a Carnot group $(G,V_1,\langle\cdot,\cdot\rangle_1)$ is the left-invariant volume form induced from the scalar product $\langle\cdot,\cdot\rangle_{\g}$ on $\g\simeq T_eG$.
\end{Definition}

{\bf The Baker--Campbell--Hausdorff formula.}
Roughly speaking,
the Baker--Campbell--Haus\-dorff formula (BCH formula) describes the solution for $Z$ to the equation $Z=\log(\exp(X_1)\times\exp(X_2))=X_1\cdot X_2$.
It links the Lie group product on $G$ and the Lie bracket on $\g$ by using combinatorial coefficients.
Namely,
for each positive integer $p\in\N$ and multi-index $(i_1,\dots,i_p)\in\{1,2\}^p$,
there is a constant $c_{(i_1,\dots,i_p)}$ such that
\[
X_1\cdot X_2=X_1+X_2+\frac{1}{2}[X_1,X_2]+\sum_{p\geq 3}\sum_{(i_1,\dots,i_p)\in \{1,2\}^p}c_{(i_1,\dots,i_p)}[X_{i_1},\dots,X_{i_p}].
\]

In this paper,
we will use the following formulae which are obtained by applying the BCH formula several times.
Set $\mathbb{I}_j^p=\{1,\dots,j\}^p$ for positive integers $j,p\in\N$:
\begin{itemize}\itemsep=0pt
 \item For each $(i_1,\dots,i_p)\!\in\!\mathbb{I}_{2k}^p$,
 there is a constant $\alpha_{(i_1,\dots,i_p)}$ such that for any vectors $X_1,\dots,X_{2k}$ with $X_i,X_{i+k}\in V_i$, $i=1,\dots,k$,
 \begin{equation}
 \Bigg(\sum_{i=1}^kX_i\Bigg)\Bigg(\sum_{i=1}^k X_{i+k}\Bigg)=\sum_{i=1}^{2k}X_i+\sum_{p\geq 2}\sum_{(i_1,\dots,i_p)\in \mathbb{I}_{2k}^p}\alpha_{(i_1,\dots,i_p)}[X_{i_1},\dots,X_{i_p}].\label{BCH1}
 \end{equation}
 \item For each $(n_1,\dots,n_p)\in\mathbb{I}_N^p$,
 there is a constant $\beta_{(n_1,\dots,n_p)}$ such that for any vectors $X_1,\dots,\allowbreak X_N\in\g$,
\begin{equation}
 X_1\cdots X_N
 =\sum_{n=1}^NX_n+\sum_{p\geq 2}\sum_{(n_1,\dots,n_p)\in\mathbb{I}_N^p}\beta_{(n_1,\dots,n_p)}[X_{n_1},\dots,X_{n_p}].\label{BCH2}
\end{equation}
\item For each $(i_1,\dots,i_p)\in\mathbb{I}_j^{p}$,
there is a constant $\gamma_{(i_1,\dots,i_p)}$ such that for any vectors $X_1,\dots,\allowbreak X_j\in\g$,
\begin{equation}
[X_1,\dots,X_j]_c
=[X_1,\dots,X_j]+\sum_{p\geq j+1}\sum_{(i_1,\dots,i_p)\in\mathbb{I}_j^p}\gamma_{(i_1,\dots,i_p)}[X_{i_1},\dots,X_{i_p}].\label{BCHfinal}
\end{equation}
\end{itemize}
Here we write the commutator of the group by $[x,y]_c=xyx^{-1}y^{-1}$,
and define the map $\psi_n{:}$ $G^n\to G$ by
\begin{gather*}
\psi_2(x_1,x_2)=[x_1,x_2]_c,
\\
\psi_n(x_1,\dots,x_n)=[x_1,\psi_{n-1}(x_2,\dots,x_n)]_c, \qquad i\geq 3.
\end{gather*}
We shortly wrote
\[
[x_1,\dots,x_n]_c:=\psi_n(x_1,\dots,x_n).
\]
Equation (\ref{BCHfinal}) is proved by induction on $j$.

\begin{Remark}\qquad
\begin{itemize}\itemsep=0pt
 \item The constants $\alpha_{(i_1,\dots,i_p)}$,
$\beta_{(n_1,\dots,n_p)}$ and $\gamma_{(i_1,\dots,i_p)}$ depend only on the choice of indices.
\item The commutator $[\cdot,\cdot]_c$ acts on a group $G$,
while the Lie bracket $[\cdot,\cdot]$ acts on its Lie algebra~$\g$.
\end{itemize}
\end{Remark}

{\bf A set of horizontal vectors.}
Let $\{X_{ni}\}_{n=1,\dots,N,i=1,\dots,j}$ be a set of (multi-indexed) horizontal vectors,
and $\mathcal{X}=\left\{\{X_{ni}\}_{n=1,\dots,N,i=1,\dots,j}\mid N,j\in\N\right\}$ a family of those sets.
We always assume $j\leq k$ for $k$-step Carnot groups.
Sometimes we simply write $\{X_{ni}\}$.
Throughout the paper,
a~set of horizontal vectors plays a central role.

From a given set of horizontal vectors,
we introduce the following three notions.

\begin{Definition}
 Let $\{X_{ni}\}_{n=1,\dots,N,\,i=1,\dots,j}$ be a set of horizontal vectors.
\begin{itemize}\itemsep=0pt
 \item[$1.$] Define the group element $y(\{X_{ni}\})$ in $G$ by
 \[
 y(\{X_{ni}\})=\prod_{n=1}^N [X_{n1},\dots,X_{nj}]_c.
 \]

 \item[$2.$] Define the vector $Y(\{X_{ni}\})$ in $\g$ by
 \[
 Y(\{X_{ni}\})=\sum_{n=1}^N[X_{n1},\dots,X_{nj}].
 \]
 \item[$3.$] Define the function $\mathsf{d}_{\mathrm{com}}\colon \mathcal{X}\to\R$ by
 \[
 \mathsf{d}_{\mathrm{com}}(\{X_{ni}\})=\sum_{n=1}^N\sum_{i=1}^j\|X_{ni}\|_1.
 \]
 We call it the combinatorial distance function.
\end{itemize}
\end{Definition}

\begin{Remark}\label{rmk2step}
 For a given set of horizontal vectors $\{X_{ni}\}$,
the group element $y(\{X_{ni}\})$ coincides with the vector $Y(\{X_{ni}\})$ via the exponential map if $G$ is abelian or $2$-step.
\end{Remark}

We denote by $\mathsf{d}_{cc}(X)$ the Carnot--Carath\'eodory distance from the identity element $e$ to a~group element $X\in\g\simeq G$.
The function $\mathsf{d}_{\mathrm{com}}$ gives an upper bound of the Carnot--Carath\'eodory distance $\mathsf{d}_{cc}$ in the following sense.
\begin{Lemma}\label{lemcom}
If a Carnot group $G$ is $k$-step,
then for any set of horizontal vectors $\{X_{ni}\}$,
we have
\[
\mathsf{d}_{cc}(y(\{X_{ni}\}))\leq 2^{k-1}\mathsf{d}_{\mathrm{com}}(\{X_{ni}\}).
\]
\end{Lemma}

\begin{proof}
 By the triangle inequality and the left-invariance of the Carnot--Carath\'eodory distance,
 we have
 \begin{align}
 \mathsf{d}_{cc}(y(\{X_{ni}\}_{n=1,\dots,N,i=1,\dots,j}))&=\mathsf{d}_{cc}
 \Bigg(\prod_{n=1}^N[X_{n1},\dots,X_{nj}]_c\Bigg) \notag \\
 &\leq \sum_{n=1}^N\mathsf{d}_{cc}([X_{n1},\dots,X_{nj}]_c).\label{eqyntri}
 \end{align}
 To compute $[X_{n1},\dots,X_{nj}]_c$,
 we observe the case $j=3$.
By the definition,
\begin{align*}
\begin{split}
 [X_{n1},X_{n2},X_{n3}]_c&=[X_{n1},[X_{n2},X_{n3}]_c]_c
 \\
 &=X_{n1}\cdot[X_{n2},X_{n3}]_c\cdot X_{n1}^{-1}\cdot[X_{n2,X_{n3}}]_c^{-1}
 \\
 &=X_{n1}\cdot \bigl(X_{n2}\cdot X_{n3}\cdot X_{n2}^{-1}\cdot X_{n3}^{-1}\bigr)\cdot X_{n1}^{-1}\cdot \bigl(X_{n3}\cdot X_{n2}\cdot X_{n3}^{-1}\cdot X_{n2}^{-1}\bigr).
 \end{split}
\end{align*}
The final term is the product of two $X_{n1}^{\pm1}$,
four $X_{n2}^{\pm 1}$ and four $X_{n3}^{\pm 1}$.
In the same way,
the group element $[X_{n1},\dots X_{nj}]$ is the product of $X_{n1}^{\pm 1},\dots,X_{nj}^{\pm 1}$.
 For $i=1,\dots,j-1$,
 the group ele\-ment~$X_{ni}^{\pm 1}$ appears $2^i$ times in this product.
 For $i=j$,
 the group element $X_{nj}$ appears $2^{j-1}$ times in this product.
Therefore the triangle inequality shows that
 \begin{gather}\label{eqytri}
 \mathsf{d}_{cc}([X_{n1},\dots,X_{nj}]_c)\leq \sum_{i=1}^{j-1}2^i\mathsf{d}_{cc}(X_{ni})+2^{j-1}\mathsf{d}_{cc}(X_{nj}).
 \end{gather}
Since vectors $X_{ni}$ are horizontal,
we have $\mathsf{d}_{cc}(X_{ni})\leq\|X_{ni}\|_1$.
Moreover,
since $i\leq j\leq k$,
we have
\begin{align*}
 \mathsf{d}_{cc}([X_{n1},\dots,X_{nj}]_c)
 &\leq\sum_{n=1}^N\Bigg[\sum_{i=1}^{j-1}2^i\mathsf{d}_{cc}(X_{ni})+2^{j-1}\mathsf{d}_{cc}(X_{nj})\Bigg]
 \\
& \leq\sum_{n=1}^N\Bigg[\sum_{i=1}^{j-1}2^i\|X_{ni}\|_1+2^{j-1}\|X_{nj}\|_1\Bigg]
\\
& \leq 2^{k-1}\sum_{n=1}^N\sum_{i=1}^j\|X_{ni}\|_1
\\
& = 2^{k-1}\mathsf{d}_{\mathrm{com}}(\{X_{ni}\}).\tag*{\qed}
\end{align*}
\renewcommand{\qed}{}
\end{proof}

\begin{Remark}
 It is difficult to replace $y(\{X_i\})$ with $Y(\{X_i\})$ in Lemma \ref{lemcom}.
 Indeed,
 we use the triangle inequality to prove the inequality (\ref{eqyntri}) and (\ref{eqytri}).
 The triangle inequality is applied to the product a sequence of group elements $X_{ni}$.
 However, the triangle inequality cannot be applied to the Lie algebraic sum.
 This is why we introduce the group element $y(\{X_{ni}\})$.
\end{Remark}

\section{Proof of the main theorem}

\subsection{2-step case}

We start from $2$-step Carnot groups.
Recall that here $\g\simeq G$ via the exponential map,
and each layer $V_i\subset \g$ is equipped with the scalar product $\langle\cdot,\cdot\rangle_i$ defined in the definition of Popp's volume.
Denote by $B^{d_i}(R)$ the ball centered at $0$ of radius $R$ of the inner metric space $(V_i,\langle\cdot,\cdot\rangle_i)$.
Via the identification $\g\simeq G$,
we identify the direct product of balls $B^{d_1}(R_1)\times B^{d_2}(R_2)\subset V_1\oplus V_2=\g$ to a subset in $G$.

\begin{Theorem}\label{propstep2}
 There exists positive constants $\epsilon_1$ and~$\epsilon_2$,
 which depend only on $d_1$,
 such that for any $2$-step Carnot group $(G,V_1,\langle\cdot,\cdot\rangle_1)$ with $d_i=\dim V_i$,
 \[
 B^{d_1}(\epsilon_1)\times B^{d_2}(\epsilon_2)\subset B_{cc}(1).
 \]
 In particular,
 the volume of the unit ball is greater than $\epsilon_1^{d_1}\epsilon_2^{d_2}\omega_{d_1}\omega_{d_2}$.
\end{Theorem}

\begin{proof}
Let $Z_2$ be a given vector in $V_2$,
and put $\nu=\|Z_2\|_2$.

 Consider $\{X_1,\dots,X_{d_1}\}$ an orthonormal (for $\langle\cdot,\cdot\rangle_1$) basis of $V_1$.
 Choose $u\in \phi_2^{-1}(Z_2)\subset V_1\otimes V_1$ so that $\nu=\|Z_2\|_2=\|u\|_{\otimes 2}$,
 where $\|\cdot\|_{\otimes 2}$ is the norm on $V_1\otimes V_1$ induced by $\langle\cdot,\cdot\rangle_{\otimes 2}$.
 By definition $\nu=\|Z_2\|_2=\|u\|_{\otimes 2}$,
 with
 \[
 u=\sum_{s,t=1}^{d_1}\alpha_{\mathrm{st}}X_s\otimes X_t=\sum_{s,t=1}^{d_1}\big(\alpha_{\mathrm{st}}^1X_s\big)\otimes\big(\alpha_{\mathrm{st}}^2 X_t\big)=:\sum_{n=1}^{d_1^2}X_{n1}\otimes X_{n2},
 \]
 where $\alpha_{\mathrm{st}}^1$ and $\alpha_{\mathrm{st}}^2$ are chosen in such a way that $\big\|\alpha_{\mathrm{st}}^1X_s\big\|_1=\big\|\alpha_{\mathrm{st}}^2X_t\big\|_1$,
 and where we rename the multi-index $(s_1,s_2)\in\{1,\dots,d_1\}^2$ the index $n\in\big\{1,\dots,d_1^2\big\}$.
 In this way, we obtain a set of horizontal vectors $\{X_{ni}\}_{n=1,\dots,d_1^2,i=1,2}$ such that the following three properties hold:
\begin{gather}
 Z_2=\sum_{n=1}^{d_1^2}[X_{n1},X_{n2}],\label{assumption1of2}
 \\
 \nu=\sqrt{\sum_{n,m=1}^{d_1^2}\langle X_{n1},X_{m1}\rangle_1\langle X_{n2},X_{m2}\rangle_1}=\sqrt{\sum_{n=1}^{d_1^2}\|X_{n1}\|_1^2\|X_{n2}\|_1^2},
 \label{assumption2of2}
 \\
\|X_{n1}\|_1=\|X_{n2}\|_1\qquad\text{for all} \quad n=1,\dots,d_1^2.\label{assumption3of2}
\end{gather}
Indeed,
equation (\ref{assumption1of2}) holds since $u\!\in\! \phi_2^{-1}(Z_2)$.
Since $\{X_1,\dots,X_{d_1}\}$ is orthonormal in $(V_1,\langle\cdot,\cdot\rangle_1)$,
$\{X_{n1}\otimes X_{n2}\}_{n=1,\dots,d_1^2}$ is an orthogonal subset in $(V_1\otimes V_1,\langle\cdot,\cdot\rangle_{\otimes 2})$,
which implies (\ref{assumption2of2}).
Equa\-tion~(\ref{assumption3of2}) holds by the choice of $\alpha_{\mathrm{st}}^1$ and $\alpha_{\mathrm{st}}^2$.

By equation (\ref{assumption3of2}),
\begin{gather*}
\mathsf{d}_{\mathrm{com}}(\{X_{ni}\})=\sum_{n=1}^{d_1^2}\|X_{n1}\|_1+\|X_{n2}\|_1 =2\sum_{n=1}^{d_1^2}\|X_{n1}\|_1,
\\
\nu=\sqrt{\sum_{n=1}^{d_1^2}\|X_{n1}\|_1^2\|X_{n2}\|_1^2}=\sqrt{\sum_{n=1}^{d_1^2}\|X_{n1}\|_1^4}.
\end{gather*}
Combined with the above two equalities and H\"older's inequality,
we obtain an upper bound of $\mathsf{d}_{\mathrm{com}}(\{X_{ni}\})$ by
\begin{equation}\label{eq2step2}
 \mathsf{d}_{\mathrm{com}}(\{X_{ni}\})= 2\sum_{n=1}^{d_1^2}\|X_{n1}\|_1\leq 2d_1^{\frac{3}{2}}\sqrt[4]{\sum_{n=1}^{d_1^2}\|X_{n1}\|_1^4}=2d_1^{\frac{3}{2}}\sqrt{\nu}.
\end{equation}

Set $\epsilon_1=\frac{1}{2}$ and $\epsilon_2=\frac{1}{64d_1^3}$.
We show that if a vector $Z=Z_1+Z_2$ $(Z_i\in V_i)$ satisfies $\|Z_i\|_i\leq \epsilon_i$,
then $\mathsf{d}_{cc}(Z)\leq 1$.

Note that $Z_1\cdot Z_2=Z_1+Z_2$ since $Z_2$ is in the center of $G$.
As we see in Remark \ref{rmk2step},
$Z_2=Y(\{X_{ni}\})=y\left(\{X_{ni}\}\right)$.
Hence by Lemma \ref{lemcom} and the inequality (\ref{eq2step2}),
\begin{align*}
 \mathsf{d}_{cc}(Z)&=\mathsf{d}_{cc}(Z_1\cdot Z_2)
=\mathsf{d}_{cc}(Z_1\cdot Z_2)
\\
 &\leq \mathsf{d}_{cc}(Z_1)+\mathsf{d}_{cc}(Z_2)
\leq \|Z_1\|_1+2\mathsf{d}_{\mathrm{com}}(\{X_{ni}\})
 \leq \frac{1}{2}+\frac{1}{2}=1.
\end{align*}

By using the volume of the unit ball in the $d$-dimensional Euclidean space $\omega_d$,
we have
\[
\operatorname{vol}(B_{cc}(1))\geq \operatorname{vol}\big(B^{d_1}(\epsilon_1)\times B^{d_2}(\epsilon_2)\big)=\frac{1}{2^{d_1+8d_2}}\frac{1}{d_1^{3d_2}}\,\omega_{d_1}\omega_{d_2}.\tag*{\qed}
\]
\renewcommand{\qed}{}
\end{proof}

\subsection{Higher step cases}
For a $k$-step Carnot group,
the following theorem implies the
 main Theorem~\ref{thmmain}.
\begin{Theorem}\label{propmain}
 There exist positive constants $\epsilon_1,\dots,\epsilon_k$,
 which depend only on $d_1$ and $k$,
 such that for any $k$-step Carnot group $(G,V_1,\langle\cdot,\cdot\rangle_1)$ with $d_i=\dim V_i$,
 \[\prod_{i=1}^kB^{d_i}(\epsilon_i)\subset B_{cc}(1).\]
 In particular,
 the volume of the unit ball is greater than $\prod_{i=1}^k\epsilon_i^{d_i}\omega_{d_i}$.
\end{Theorem}

Compared from the $2$-step case,
its difficulty is the non-coincidence of $y(\{X_{ni}\})$ and $Y(\{X_{ni}\})$.
We will compute its difference by using the BCH formula.

Let $Z_j$ be a given vector in $V_j$,
and put $\nu_j=\|Z_j\|_j$.
 Consider $\{X_1,\dots,X_{d_1}\}$ an orthonormal (for $\langle\cdot,\cdot\rangle_1$) basis of $V_1$.
 Choose $u_j\in \phi_j^{-1}(Z_2)\subset V_1^{\otimes j}$ so that $\nu_j=\|Z_j\|_j=\|u_j\|_{\otimes j}$,
 where $\|\cdot\|_{\otimes j}$ is the norm on $V_1^{\otimes j}$ induced by $\langle\cdot,\cdot,\rangle_{\otimes j}$.
 By definition $\nu_j=\|Z_j\|_j=\|u_j\|_{\otimes j}$,
 with
 \begin{align*}
 u_j=\sum_{s_1,\dots,s_j=1}^{d_1}\alpha_{s_1\dots s_j}X_{s_1}\otimes\cdots\otimes X_{s_j}
 &=\sum_{s_1,\dots,s_j=1}^{d_1}\big(\alpha_{s_1\dots s_j}^1X_{s_1}\big)\otimes\cdots\otimes \big(\alpha_{s_1\dots s_j}^j X_{s_j}\big)
 \\
 &=:\sum_{n=1}^{d_1^j}X_{n1}\otimes \cdots\otimes X_{nj},
 \end{align*}
 where $\alpha_{s_1\dots s_j}^i$, $i=1,\dots,j$, are chosen in such a way that $\big\|\alpha_{s_1\dots s_j}^1X_{s_1}\big\|_1=\cdots =\big\|\alpha_{s_1\dots s_j}^jX_{s_j}\big\|_1$,
 and where we rename the multi-index $(s_1,\dots,s_j)\in\{1,\dots,d_1\}^j$ the index $n\in\big\{1,\dots,d_1^j\big\}$.
 In~this way,
we obtain a set of horizontal vectors $\{X_{ni}\}_{n=1,\dots,d_1^j,\,i=1,\dots,j}$ such that the following three properties holds:
 \begin{gather}
Z_j=\sum_{n=1}^{d_1^j}[X_{n1},\dots,X_{nj}],\label{xni1}
 \\
\nu_j=\sqrt{\sum_{n,m=1}^{d_1^j}\langle X_{n1},X_{m1}\rangle_1\cdots\langle X_{nj},Y_{nj}\rangle_1}=\sqrt{\sum_{n=1}^{d_1^j}\|X_{n1}\|_1^2\cdots\|X_{nj}\|_1^2},\label{xni2}
 \\
\|X_{n1}\|_1=\cdots=\|X_{nj}\|_1\qquad \text{for all}\quad n=1,\dots,d_1^j.\label{xni3}
\end{gather}

\begin{Definition}
 We say that a set of horizontal vectors $\{X_{ni}\}$ is \textit{adjusted to $Z_j$} if the three conditions (\ref{xni1}),
 (\ref{xni2}) and (\ref{xni3}) hold.
\end{Definition}

Let us consider the difference between the group element $y(\{X_{ni}\})$ and the vector $Y(\{X_{ni}\})$ of a set of horizontal vectors adjusted to $Z_j$.
Define the map $P_l\colon \g\to V_l$ to be the linear projection.
By the equation (\ref{BCHfinal}),
\[
 P_l(y(\{X_{ni}\}))=\begin{cases}
 0, & l\leq j-1,\\
Z_j=Y(\{X_{ni}\}), & l=j,\\
\text{a possibly non-zero vector}, & l\geq j+1.
\end{cases}
\]
We label a possibly non-zero vector $P_l(y(\{X_{ni}\}))$, $l\geq j+1$, as follows.

\begin{Definition}
 Denote by $A_l(\{X_{ni}\})$ the image of $y(\{X_{ni}\})$ by $P_l$.
 We call the vec\-tor~$A_l(\{X_{ni}\})$ the $l$-error vector of $\{X_{ni}\}$.
\end{Definition}
Sometimes we simply write $A_l$.
We will give an upper bound of $\|A_l\|_l$ later in Lemma \ref{lemvjbound}.

Let us start from the preparation.
\begin{Lemma}\label{lemsubmulti}
 For any two vectors $Z_p\in V_p$ and $Z_q\in V_q$,
 we have
 \[\|[Z_p,Z_q]\|_{p+q}\leq 2^{p\wedge q}\|Z_p\|_p\|Z_q\|_q,\]
 where $p\wedge q=\min\{p,q\}$.
\end{Lemma}

\begin{proof}
 {\sloppy By the skew-symmetry of the Lie bracket,
 we can assume $p\leq q$.
 Let $\bigl\{X_{ni}^{(p)}\bigr\}$ $\bigl(\text{resp.}~\bigl\{X_{mj}^{(q)}\bigr\}\bigr)$ be a set of horizontal vectors adjusted to $Z_p$ $(\text{resp.}~Z_q)$.
From the bi-linearity of the Lie bracket and the subadditivity of the norm $\|\cdot\|_{p+q}$,
we have
\begin{align}
 \|[Z_p,Z_q]\|_{p+q}=\bigg\|\bigg[\sum_{n=1}^{d_1^p}\big[X_{n1}^{(p)},\dots,X_{np}^{(p)}\big], \sum_{m=1}^{d_1^q}\big[X_{m1}^{(q)},\dots,X_{mq}^{(q)}\big]\bigg]\bigg\|_{p+q} \notag \\
 \leq\sum_{n=1}^{d_1^p}\sum_{m=1}^{d_1^q}\big\|\big[\big[X_{n1}^{(p)},\dots,X_{np}^{(p)}\big], \big[X_{m1}^{(q)},\dots,X_{mq}^{(q)}\big]\big]\big\|_{p+q}.\label{symm}
\end{align}}\noindent
By applying the Jacobi identity $[[X,Y],Z]=[X,[Y,Z]]-[Y,[X,Z]]$ several times,
we can rewrite
\begin{gather*}
\bigl[\big[X_{n1}^{(p)},\dots,X_{np}^{(p)}\big],\big[X_{m1}^{(q)},\dots,X_{mq}^{(q)}\big]\bigr]
 \\ \qquad
{} =\sum_{\sigma\in S}\epsilon_{\sigma}\big[X_{n\sigma(1)}^{(p)},\big[X_{n\sigma(2)}^{(p)},\dots, \big[X_{n\sigma(p)}^{(p)},\big[X_{m1}^{(q)},\dots,X_{mq}^{(q)}\big]\big],\ldots\big]\big],
\end{gather*}
where $S$ is the subset of the symmetric group of degree $p$ consisting of $\sigma$ such that if $\sigma(a)=p$,
then
\[
\sigma(1)<\sigma(2)<\cdots<\sigma(a)>\sigma(a+1)>\cdots>\sigma(p),
\]
and $\epsilon_{\sigma}\in\{\pm1\}$.
Since the size of the subset $S$ is $2^p$,
we can compute an upper bound of (\ref{symm}) by
\begin{gather*}
 \sum_{n=1}^{d_1^p}\sum_{m=1}^{d_1^q}\bigg\|\sum_{\sigma\in S}\epsilon_{\sigma}\big[X_{n\sigma(1)}^{(p)},\big[X_{n\sigma(2)}^{(p)},\dots,\big[X_{n\sigma(p)}^{(p)}, \big[X_{m1}^{(q)},\dots,X_{mq}^{(q)}\big]\big],\ldots\big]\big]\bigg\|_{p+q}
 \\ \qquad
{}\leq\sum_{n=1}^{d_1^p}\sum_{m=1}^{d_1^q}2^p\prod_{i=1}^p\big\|X_{ni}^{(p)}\big\|_1\prod_{j=1}^q\big\|X_{mj}^{(q)}\big\|_1
 =2^p\Bigg(\sum_{n=1}^{d_1^p}\prod_{i=1}^p\big\|X_{ni}^{(p)}\big\|_1\Bigg) \Bigg(\sum_{m=1}^{d_1^q}\prod_{j=1}^q\big\|X_{mj}^{(q)}\big\|_1\Bigg)
 \\ \qquad
 {}\leq2^p\|Z_p\|_p\|Z_q\|_q.\tag*{\qed}
\end{gather*}
\renewcommand{\qed}{}
\end{proof}

Next we will control an upper bound of the $\|A_l(\{X_{ni}\})\|_l$.

\begin{Lemma}\label{lemvjbound}
 For $j=1,\dots,k$,
 there is a positive constant $\theta_j=\theta_j(d_1,j,k)$ such that for any vector $Z_j\in V_j$, $\nu_j=\|Z_j\|_j$, and any set of horizontal vectors $\{X_{ni}\}$ adjusted to $Z_j$,
 \[\|A_l(\{X_{ni}\})\|_l\leq \theta_j\nu_j^{\frac{l}{j}}.\]
\end{Lemma}

\begin{proof}
For each $n=1,\dots,d_1^j$,
let $U_n=y(\{X_{ni}\}_{i=1,\dots,j})$.
By the equation (\ref{BCHfinal}),
 \begin{equation}\label{equn}
 U_n=[X_{n1},\dots,X_{nj}]+\sum_{m\geq j+1}\sum_{(i_1,\dots,i_m)\in\mathbb{I}_j^m}\gamma_{(i_1,\dots,i_m)}[X_{ni_1},\dots,X_{ni_m}].
 \end{equation}

By using the equation (\ref{BCH2}),
the product of the group elements $U_n$ is written by
\[
 \prod_{n=1}^{d_1^j}U_n=Z_j+\sum_{q\geq2}\sum_{(n_1,\dots,n_q)\in \mathbb{I}_{d_1^j}^q}\beta_{(n_1,\dots,n_q)}[U_{n_1},\dots,U_{n_q}].
\]
Notice that the group element $\prod_{n=1}^{d_1^j}U_n$ coincides with $y(\{X_{ni}\})=Z_j+\sum_{l=j+1}^{k}A_l$ from its definition.
Thus we can explicitly write $A_l$ by
\[
A_l=\sum_{q\geq 2}\sum_{(n_1,\dots,n_q)\in\mathbb{I}_{d_1^j}^q}\sum_{m_1+\cdots +m_q=l}\beta_{(n_1,\dots,n_q)}[P_{m_1}(U_{n_1}),\dots,P_{m_q}(U_{n_q})].
\]

From (\ref{equn}),
we can compute $\|P_m(U_n)\|_m$ by
\[
\|P_m(U_n)\|_m\leq\begin{cases}
 0, & m=1,\dots,j-1,\\
\|X_{n1}\|_1^j, & m=j,\\
\sum_{(i_1,\dots,i_m)\in\mathbb{I}_j^m}|\gamma_{(i_1,\dots,i_m)}|
\|X_{n1}\|_1^m, & m=j+1,\dots,k.
\end{cases}
\]
In any case,
$\|P_m(U_n)\|_m$ is less than or equal to $\tilde{\gamma}\nu_j^{\frac{m}{j}},$
where $\tilde{\gamma}\!=\!\max\big\{1,\sum_{(i_1,\dots,i_m)\in \mathbb{I}_j^m}\!|\gamma_{(i_1,\dots,i_m)}|\big\}$.
By the subadditivity of the norm $\|\cdot\|_m$ and Lemma \ref{lemsubmulti},
we have
\begin{align*}
 \|A_l\|_l&\leq\sum_{q\geq2}\sum_{(n_1,\dots,n_{q})\in\mathbb{I}_{d_1^j}^q}\sum_{m_1+\cdots+m_q=l}2^l  |\beta_{(n_1,\dots,n_q)}|  \|P_{m_1}(U_{n_1})\|_{m_1}\cdots\|P_{m_q}(U_{n_q})\|_{m_q}
 \\
 &\leq \sum_{q\geq2}\sum_{(n_1,\dots,n_q)\in \mathbb{I}_{d_1^j}^q}\sum_{m_1+\cdots+m_q=l}2^l  |\beta_{(n_1,\dots,n_q)}| \big(\tilde{\gamma}\nu_j^{\frac{m_1}{j}}\big)\cdots \big(\tilde{\gamma}\nu_j^{\frac{m_q}{j}}\big)
 \\
 &= \sum_{q\geq2}\sum_{(n_1,\dots,n_q)\in \mathbb{I}_{d_1^j}^q}\sum_{m_1+\cdots+m_q=l} 2^l  |\beta_{(n_1,\dots,n_q)}|  \tilde{\gamma}^q  \nu_j^{\frac{l}{j}}
 \\
 &\leq \sum_{q\geq2}\sum_{(n_1,\dots,n_q)\in \mathbb{I}_{d_1^j}^q}\binom{l+q-1}{q-1}   2^l  |\beta_{(n_1,\dots,n_q)}|  \tilde{\gamma}^q  \nu_j^{\frac{l}{j}}
 \\
 &\leq\sum_{q\geq2}\sum_{(n_1,\dots,n_q)\in \mathbb{I}_{d_1^j}^q} \binom{l+q-1}{q-1}  2^l  \tilde{\beta}  \tilde{\gamma}^q  \nu_j^{\frac{l}{j}}
 \\
 &\leq\sum_{q\geq2}d_1^{jq}  \binom{l+q-1}{q-1}  2^l  \tilde{\beta}  \tilde{\gamma}^q  \nu_j^{\frac{l}{j}}
 \leq\sum_{q\geq2}d_1^{jk}  \binom{l+q-1}{q-1}  2^l  \tilde{\beta}  \tilde{\gamma}^k  \nu_j^{\frac{l}{j}}
 \\
 &\leq d_1^{jk}  2^{l+k-1}  2^l  \tilde{\beta}  \tilde{\gamma}^k  \nu_j^{\frac{l}{j}}
 \leq 8^k  d_1^{jk} \tilde{\beta}  \tilde{\gamma}^k  \nu_j^{\frac{l}{j}},
\end{align*}
where $\tilde{\beta}=\max\big\{|\beta_{(n_1,\dots,n_q)}|\mid (n_1,\dots,n_q)\in \mathbb{I}_{d_1^j}^q,\,q=2,\dots,k\big\}$.

Notice that the constants $\tilde{\beta}$ and $\tilde{\gamma}$ depend only on the constants $\beta_{(n_1,\dots,n_q)}$ and $\gamma_{(i_1,\dots,i_m)}$ in the BCH formulas (\ref{BCH2}) and (\ref{BCHfinal}).
Since $m$ and $q$ are not greater than $k$,
they ultimately depend on $d_1$, $j$ and $k$.
Hence we have obtained a desired constant $\theta_j$ by letting
$\theta_j=8^kd_1^{jk}\tilde{\beta}\tilde{\gamma}^k.$
\end{proof}

The norm $\|\cdot\|_j$ controls the combinatorial distance $\mathsf{d}_{\mathrm{com}}$ as follows.

\begin{Lemma}\label{lemCC}
For any $Z_j\in V_j$ with $\nu_j=\|Z_j\|_j$ and any set of horizontal vector $\{X_{ni}\}$ adjusted to $Z_j$,
there is an upper bound of the combinatorial distance given by
\[
\mathsf{d}_{\mathrm{com}}(\{X_{ni}\})\leq jd_1^{\frac{2j-1}{2}}\nu_j^{\frac{1}{j}}.
\]
\end{Lemma}

\begin{proof}
By the assumption (\ref{xni3})
we obtain
\[
 \mathsf{d}_{\mathrm{com}}(\{X_{ni}\})=j \sum_{n=1}^{d_1^j}\|X_{n1}\|_1.
\]
By H\"older's inequality,
\[j \sum_{n=1}^{d_1^j}\|X_{n1}\|_1\leq jd_1^{\frac{2j-1}{2}}\sqrt[2j]{\sum_{n=1}^{d_1^j}\|X_{n1}\|_1^{2j}}.
\]
By the equality (\ref{xni2}),
\[
 jd_1^{\frac{2j-1}{2}}\sqrt[2j]{\sum_{n=1}^{d_1^j}\|X_{n1}\|_1^{2j}}=jd_1^{\frac{2j-1}{2}}\nu_j^{\frac{1}{j}}.
\]
The above three inequalities prove the lemma.
\end{proof}

We have considered a set of vectors adjusted to a vector $Z_j$ in each layer $V_j$.
We will introduce a similar notion for a vector $Z$ in the whole Lie algebra $\g$.

Let $Z=\sum_{j=1}^k Z_j$ be a vector in $\g=\bigoplus_{j=1}^k V_j$, $Z_j\in V_j$.
We will inductively define a set of vectors $\bigl\{X_{ni}^{(j)}\bigr\}_{n=1,\dots,d_1^j,\,i=1,\dots,j}$ for $j=1,\dots,k$ as follows.
For $Z_1\in V_1$,
define a set of vectors $\bigl\{X_{ni}^{(1)}\bigr\}=\{Z_1,0,\dots,0\}$.
Let $\bigl\{X_{ni}^{(2)}\bigr\}$ be a set of horizontal vectors adjusted to $Z_2$.
Then for a~couple of sets of horizontal vectors $\bigl(\bigl\{X_{ni}^{(1)}\bigr\},\bigl\{X_{ni}^{(2)}\bigr\}\bigr)$,
there are vectors $B_l^{(2)}\bigl(\bigl\{X_{ni}^{(1)}\bigr\},\bigl\{X_{ni}^{{(2)}}\bigr\}\bigr)\allowbreak\in V_l$ for $l>2$ such that
\[
y\bigl(\bigl\{X_{ni}^{(1)}\bigr\}\bigr)\cdot y\bigl(\bigl\{X_{ni}^{(2)}\bigr\}\bigr) =Z_1+Z_2+\sum_{l=3}^kB_l^{(2)}\bigl(\bigl\{X_{ni}^{(1)}\bigr\},\bigl\{X_{ni}^{(2)}\bigr\}\bigr).
\]
Next let $\bigl\{X_{ni}^{(3)}\bigr\}$ be a set of horizontal vectors adjusted to $Z_3-B_3^{(2)}$.
Then for a triple of sets of horizontal vectors $\bigl(\bigl\{X_{ni}^{(1)}\bigr\},\bigl\{X_{ni}^{(2)}\bigr\},\bigl\{X_{ni}^{(3)}\bigr\}\bigr)$,
there are vectors $B_l^{(3)}\bigl(\bigl\{X_{ni}^{(1)}\bigr\},\bigl\{X_{ni}^{(2)}\bigr\},\bigl\{X_{ni}^{(3)}\bigr\}\bigr)\allowbreak\in V_l$ for $l>3$ such that
\[
\prod_{j=1}^3y\big(\bigl\{X_{ni}^{(j)}\bigr\}\big)=\sum_{l=1}^3Z_l+\sum_{l=4}^k B_l^{(3)}\bigl(\bigl\{X_{ni}^{(1)}\bigr\},\bigl\{X_{ni}^{(2)}\bigr\},\bigl\{X_{ni}^{(3)}\bigr\}\bigr).
\]
In this way,
we can inductively define a set of horizontal vectors $\bigl\{X_{ni}^{(j)}\bigr\}$ and error vectors $B_l^{(j)}\bigl(\bigl\{X_{ni}^{(1)}\bigr\},\dots,\bigl\{X_{ni}^{(j)}\bigr\}\bigr)$.
We will summarize this argument in the following definition.

\begin{Definition}
 For a vector $\sum_{j=1}^k Z_j\in\bigoplus V_j$,
 sets of horizontal vectors $\bigl\{X_{ni}^{(1)}\bigr\},\dots,\bigl\{X_{ni}^{(k)}\bigr\}$ and vectors $B_l^{(j)}\bigl(\bigl\{X_{ni}^{(1)}\bigr\},\dots,\bigl\{X_{ni}^{(j)}\bigr\}\bigr)\in V_l$, $j=1,\dots,k$, are inductively defined by
 \begin{itemize}\itemsep=0pt
 \item[(X1)] $\bigl\{X_{ni}^{(1)}\bigr\}=\{Z_1,0,\dots,0\}$,
 \item[(B1)] $B_l^{(1)}\bigl(\bigl\{X_{ni}^{(1)}\bigr\}\bigr)=0$ for $l=1,\dots,k$,
 \item[(Xj)] $\bigl\{X_{ni}^{(j+1)}\bigr\}$ is a set of horizontal vectors adjusted to $Z_{j+1}-B_{j+1}^{(j)}\big(\bigl\{X_{ni}^{(1)}\bigr\},\dots,\bigl\{X_{ni}^{(j)}\bigr\}\big)$,
 \item[(Bj)] $B_l^{(j)}\big(\bigl\{X_{ni}^{(1)}\bigr\},\dots,\bigl\{X_{ni}^{(j)}\bigr\}\big) =P_l\big(\prod_{m=1}^jy\big(\bigl\{X_{ni}^{(m)}\bigr\}\big)\big)$ for $l=j+1,\dots,k$.
 \end{itemize}
 We call $\big(\bigl\{X_{ni}^{(1)}\bigr\},\dots,\bigl\{X_{ni}^{(k)}\bigr\}\big)$ a $k$-tuple of sets of horizontal vectors adjusted to $Z$,
 and call $B_{l}^{(j)}\big(\bigl\{X_{ni}^{(1)}\bigr\},\dots,\bigl\{X_{ni}^{(j)}\bigr\}\big)$ an $(l,j)$-error vector.
\end{Definition}

We will simply write $B_l^{(j)}$.
Since $\g$ is $k$-step,
the error vectors $B_l^{(k)}$ vanishes and
\[
Z=\sum_{j=1}^k Z_j=\prod_{j=1}^ky\big(\bigl\{X_{ni}^{(j)}\bigr\}\big).
\]

\begin{Remark}
 Since the choice of a set of horizontal vectors adjusted to $Z_j\in V_j$ is not unique,
 the choice of error vectors is not unique too.
\end{Remark}

The norm of error vectors $\big\|B_l^{(j)}\big\|_l$ can be controlled as follows.

\begin{Lemma}\label{lempolynomial}
 For $Z=\sum_{j=1}^k Z_j\in\g$ with $\nu_j=\|Z_j\|_j$,
 let $\big(\bigl\{X_{ni}^{(1)}\bigr\},\dots,\bigl\{X_{ni}^{(k)}\bigr\}\big)$ be a $k$-tuple of sets of horizontal vectors adjusted to $Z$.
 For $j=1,\dots,k-1$ and $l=j+1,\dots,k$,
 there are polynomials $Q_{lj}(\beta_1,\beta_2,\dots,\beta_k)$ such that
 \begin{gather*}
 Q_{lj}(0,\dots,0)=0\qquad \text{and}\qquad
 \bigl\|B_l^{(j)}\bigr\|_l\leq Q_{lj}\big(\nu_1,\sqrt{\nu_2},\dots,\sqrt[k]{\nu_k}\big).
 \end{gather*}
 Moreover, their coefficients depend only on the dimensions $d_1$ and $k$.
\end{Lemma}

\begin{proof}
 We prove the assertion by inductions on $j$.
 When $j=1$,
 then $B_l^{(1)}=0$ for all $l=2,\dots,k$,
 so the lemma trivially holds.

 Assume that the lemma holds for $j>1$.
 From the definition of the $l$-error vector $A_l=A_l\big(\bigl\{X_{ni}^{(j+1)}\bigr\}\big)$ and the $(l,j)$-error vector $B_l^{(j)}=B_l^{(j)}\big(\bigl\{X_{ni}^{(1)}\bigr\},\dots,\bigl\{X_{ni}^{(j)}\bigr\}\big)$,
 we have
 \[
 y\big(\bigl\{X_{ni}^{(j+1)}\bigr\}\big)=Z_{j+1}+\sum_{l=j+2}^kA_l,
 \]
 and
 \[
 \prod_{l=1}^{j}y\big(\bigl\{X_{ni}^{(l)}\bigr\}\big)=\sum_{l=1}^{j}Z_l+\sum_{l=j+1}^kB_l^{(j)}.
 \]
 In particular,
 we can see that the $(l,j+1)$-error vector $B_l^{(j+1)}$ can be written by using the $l$-error vector $A_l$ and $(l,j)$-error vector $B_l^{(j)}$,
 \begin{align*}
\sum_{l=1}^{j+1}Z_l+\sum_{l=j+2}^kB_l^{(j+1)}
 &=\prod_{l=1}^{j+1}y\big(\bigl\{X_{ni}^{(l)}\bigr\}\big)
 =\Bigg(\sum_{l=1}^{j}Z_l+\sum_{l=j+1}^kB_l^{(j)}\Bigg)\Bigg(Z_{j+1}+\sum_{l=j+2}^kA_l\Bigg).
 \end{align*}
 More precisely,
 let $\{S_p\}_{p=1,\dots,2k}$ be the finite set of vectors defined by
 \[
 S_p=\begin{cases}
 Z_p, & p=1,\dots,j,
 \\
 B_p^{(j)}, & p=j+1,\dots,k,
 \\
 0, & p=k+1,\dots,k+j,
 \\
 Z_{j+1}, & p=k+j+1,
 \\
 A_{p-k}, & p=k+j+2,\dots,2k.
 \end{cases}
 \]
 Denote by $\deg(S_p)$ the number of the layer in which the vector $S_p\in\g=\bigoplus V_j$ is.
 For $p=1,\dots,2k$,
let $\tilde{Q}_p\big(\nu_1,\dots,\sqrt[k]{\nu_k}\big)$ be the polynomial which controls $\|S_p\|_p$ $(\|S_p\|_{p-k}$ for $k+1\leq p\allowbreak\leq 2k$) so that it attains zero at $(0,\dots,0)$ and the coefficients depend on $d_1$ and $k$.
We can choose such polynomials by the induction hypothesis and Lemma \ref{lemvjbound}.

 By applying the BCH formula (\ref{BCH1}),
 we have
 \begin{gather*}
 \Bigg(\sum_{l=1}^{j}Z_l+\sum_{l=j+1}^kB_l^{(j)}\Bigg)\Bigg(Z_{j+1}+\sum_{l=j+2}^kA_l\Bigg)
 \\ \qquad
{} = \Bigg(\sum_{i=1}^kS_p\Bigg)\Bigg(\sum_{i=1}^{k}S_{p+i}\Bigg)
 =\sum_{i=1}^{2k}S_i+\sum_{q\geq 2}\sum_{(p_1,\dots,p_q)\in \mathbb{I}_{2k}^q}\alpha_{(p_1,\dots,p_q)}[S_{p_1},\dots,S_{p_q}].
 \end{gather*}
 Thus the $(l,j+1)$-error vector $B_l^{(j+1)}$ can be written by
 \[
 B_l^{(j+1)}=\sum_{q=2}^k\sum_{(p_1,\dots,p_q)\in\mathbb{I}_{2k}^q,\, \deg(S_{p_1})+\cdots+\deg(S_{p_q})=l}\alpha_{(p_1,\dots,p_q)}[S_{p_1},\dots,S_{p_q}].
 \]

 By the triangle inequality and Lemma \ref{lemsubmulti},
 \[
 \big\|B_l^{(j+1)}\big\|_l\leq \sum_{q=2}^k\sum_{(p_1,\dots,p_q)\in\mathbb{I}_{2k}^q,\,\deg(S_{p_1}) +\cdots+\deg(S_{p_q})=l}|\alpha_{p_1,\dots,p_q}|2^{p_1\wedge\cdots\wedge p_q}\prod_{i=1}^q\tilde{Q}_{p_i}\big(\nu_1,\dots,\sqrt[k]{\nu_k}\big),
 \]
 where $p_1\wedge\cdots\wedge p_q$ is the minimum of $p_1,\dots,p_q$.
\end{proof}

We introduce a function $\mathsf{d}_{\mathrm{com}}^{(j)}\colon \mathcal{X}^j\to\R$ which can be regarded as the combinatorial distance function for a $j$-tuple of sets of horizontal vectors.

\begin{Definition}
 For a $j$-tuple of sets of horizontal vectors $\big(\bigl\{X_{ni}^{(1)}\bigr\},\dots,\bigl\{X_{ni}^{(j)}\bigr\}\big)$,
 define
 \[\mathsf{d}_{\mathrm{com}}^{(j)}\big(\bigl\{X_{ni}^{(1)}\bigr\},\dots,\bigl\{X_{ni}^{(j)}\bigr\}\big)= \sum_{l=1}^j\mathsf{d}_{\mathrm{com}}\big(\bigl\{X_{ni}^{(l)}\bigr\}\big).\]
\end{Definition}

Lemma \ref{lemcom} and the triangle inequality imply the following lemma.

\begin{Lemma}\label{lemcom2}
 For any $j$-tuple of sets of horizontal vectors $\big(\bigl\{X_{ni}^{(1)}\bigr\},\dots,\bigl\{X_{ni}^{(j)}\bigr\}\big)$,
 we have
 \[
 \mathsf{d}_{cc}\Bigg(\prod_{l=1}^jy\big(\bigl\{X_{ni}^{(l)}\bigr\}\big)\Bigg)\leq 2^{k-1}\mathsf{d}_{\mathrm{com}}^{(j)}\big(\bigl\{X_{ni}^{(1)}\bigr\}, \dots,\bigl\{X_{ni}^{(j)}\bigr\}\big).
 \]
\end{Lemma}

Now we are ready to prove Theorem~\ref{propmain}.
We show the following technical proposition.

\begin{Proposition}
 There are positive constants $\epsilon_1,\dots,\epsilon_k$,
 which depend only on $d_1$ and $k$,
 such that if a vector $Z=\sum_{j=1}^k Z_j$ satisfies $\|Z_j\|_j\leq \epsilon_j$,
 then there is a $k$-tuple of sets of horizontal vectors $\big(\bigl\{X_{ni}^{(1)}\bigr\},\dots,\bigl\{X_{ni}^{(k)}\bigr\}\big)$ adjusted to $Z$ such that
 \[
 \mathsf{d}_{\mathrm{com}}^{(k)}\big(\bigl\{X_{ni}^{(1)}\bigr\},\dots, \bigl\{X_{ni}^{(k)}\bigr\}\big)\leq \frac{1}{2^{k-1}}.
 \]
\end{Proposition}

This proposition implies Theorem \ref{propmain}.
Indeed,
we can check that $Z=\prod_{j=1}^ky\big(\bigl\{X_{ni}^{(j)}\bigr\}\big)$ since~$G$ is $k$-step.
Combined with Lemma \ref{lemcom2},
\[
\mathsf{d}_{cc}(Z)=\mathsf{d}_{cc}\Bigg(\prod_{j=1}^ky\big(\bigl\{X_{ni}^{(j)}\bigr\}\big)\Bigg)\leq 2^{k-1}\mathsf{d}_{\mathrm{com}}^{(k)}\big(\bigl\{X_{ni}^{(1)}\bigr\}, \dots,\bigl\{X_{ni}^{(k)}\bigr\}\big)\leq 1.
\]

\begin{proof}
 We will prove by induction on $k$.
 We have already shown the assertion for $2$-step Carnot group with $\epsilon_1=\frac{1}{2}$ and $\epsilon_2=\frac{1}{64d_1^3}$ in Theorem \ref{propstep2}.

 Assume that the assertion is true for $(k-1)$-step Carnot groups with the grading $\dim V_i=d_i$ with positive numbers $\tilde{\epsilon}_1,\dots,\tilde{\epsilon}_{k-1}$.
 Let $G_k$ be the subgroup of $G$ generated by $\{[x_1,\dots,x_k]_c\mid x_1,\dots,x_k\in G\}$.
 Notice that the quotient group $G/G_k$ is $(k-1)$-step Carnot group which has the grading $\bigoplus_{j=1}^{k-1}V_j$.
 From the induction hypothesis,
 if a vector $\sum_{j=1}^{k-1}Z_j$ satisfies $\|Z_j\|_j\leq \tilde{\epsilon}_j$,
 then there are $(k-1)$-tuple of sets of horizontal vectors $\big(\bigl\{X_{ni}^{(1)}\bigr\},\dots,\bigl\{X_{ni}^{(k-1)}\bigr\}\big)$ adjusted to $\sum_{j=1}^{k-1}Z_j$ such that
 \[
 \mathsf{d}_{\mathrm{com}}^{(k-1)}\big(\bigl\{X_{ni}^{(1)}\bigr\}, \dots,\bigl\{X_{ni}^{(k-1)}\bigr\}\big)\leq \frac{1}{2^{k-2}}.
 \]

 Moreover,
 for a positive number $t>0$,
 the $(k-1)$-tuple of sets of horizontal vectors $\big(\bigl\{tX_{ni}^{(1)}\bigr\},\dots,\bigl\{tX_{ni}^{(k-1)}\bigr\}\big)$ is adjusted to the vector $\sum_{j=1}^{k-1}t^jZ_j$,
 and satisfies
 \begin{equation}\label{tz2}
 \mathsf{d}_{\mathrm{com}}^{(k-1)}\big(\bigl\{tX_{ni}^{(1)}\bigr\},\dots,\bigl\{tX_{ni}^{(k-1)}\bigr\}\big)\leq \frac{t}{2^{k-2}}.
 \end{equation}

 Next we consider the product $\prod_{j=1}^{k-1}y\big(\bigl\{tX_{ni}^{(j)}\bigr\}\big)$ in the original group $G$.
 By using the \mbox{$(k,k-1)$}-error vector $B_k^{(k-1)}$,
 we can write
 \[
 \prod_{j=1}^{k-1}y\big(\bigl\{tX_{ni}^{(j)}\bigr\}\big)=\sum_{j=1}^{k-1}t^jZ_j+B_k^{(k-1)}.
 \]

By Lemma \ref{lempolynomial},
there is a polynomial $Q_{kk-1}$ such that
\[
\big\|B_k^{(k-1)}\big\|_k\leq Q_{kk-1}\bigl(t\tilde{\epsilon}_1,\dots,t\sqrt[k]{\tilde{\epsilon}_{k}}\bigr).
\]

 Now let $Z_k$ be a vector in $V_k$,
 $\nu_k=\|Z_k\|_k$,
 and $\bigl\{X_{ni}^{(k)}\bigr\}$ a set of horizontal vectors adjusted to $Z_k-B_k^{(k-1)}$.
 By the definition of $\mathsf{d}_{\mathrm{com}}^{(k)}$ and (\ref{tz2}),
 we have
 \begin{align*}
 \mathsf{d}_{\mathrm{com}}^{(k)}\big(\bigl\{tX_{ni}^{(1)}\bigr\}, \dots,\bigl\{tX_{ni}^{(k-1)}\bigr\},\bigl\{X_{ni}^{(k)}\bigr\}\big)
 &=\mathsf{d}_{\mathrm{com}}^{(k-1)}\big(\bigl\{tX_{ni}^{(1)}\bigr\}, \dots,\bigl\{tX_{ni}^{(k-1)}\bigr\}\big)+\mathsf{d}_{\mathrm{com}}\big(\bigl\{X_{ni}^{(k)}\bigr\}\big)
\\
& \leq \frac{t}{2^{k-2}}+\mathsf{d}_{\mathrm{com}}\big(\bigl\{X_{ni}^{(k)}\bigr\}\big).
 \end{align*}

Since the set of horizontal vectors $\bigl\{X_{ni}^{(k)}\bigr\}$ is adjusted to $Z_k-B_k^{(k-1)}$,
Lemma \ref{lemCC} yields
\begin{gather*}
\mathsf{d}_{\mathrm{com}}\big(\bigl\{X_{ni}^{(k)}\bigr\}\big)\leq kd_1^{\frac{2k-1}{2}}\big\|Z_k-B_k^{(k-1)}\big\|_k^{\frac{1}{k}}
\leq kd_1^{\frac{2k-1}{2}}\big(\nu_k+Q_{kk-1}\big(t\tilde{\epsilon}_1, \dots,t\sqrt[k]{\tilde{\epsilon}_{k}}\big)\big)^{\frac{1}{k}}.
\end{gather*}

Since the polynomial $Q_{kk-1}$ attains zero at $(0,\dots,0)\in \R^{k}$ and the coefficients depend only on $d_1$ and $k$,
there are positive numbers $T$, $\hat{\epsilon}_k$,
which depend only on the dimension $d_1$ and~$k$,
such that if $t\leq T$ and $\nu_k\leq \hat{\epsilon}_k$,
then
\[
\frac{t}{2^{k-2}}+\mathsf{d}_{\mathrm{com}}\big(\bigl\{X_{ni}^{(k)}\bigr\}\big)\leq \frac{1}{2^{k-1}}.
\]
We conclude the proposition by letting $\epsilon_j=T^j\tilde{\epsilon}_j$ for $j=1,\dots,k-1$ and $\epsilon_k=\min\big\{T^k\tilde{\epsilon}_k,\hat{\epsilon}_k\big\}$.
\end{proof}

\subsection*{Acknowledgements}

The author would appreciate to thank Professor Takumi Yokota for many insightful suggestions.
The author thanks the referees for many helpful comments on earlier drafts of the manuscript.
This research is supported by JSPS KAKENHI grant number 18K03298 and 20J13261.

\pdfbookmark[1]{References}{ref}
\LastPageEnding

\end{document}